\documentclass{amsart}

\usepackage{amsmath,amssymb,amsthm,a4wide}
\usepackage{graphicx,tikz}
\newtheorem{theorem}{Theorem}
\newtheorem*{thm}{Theorem}
\newtheorem*{proposition}{Proposition}
\newtheorem*{corollary}{Corollary}

\newtheorem{lemma}{Lemma}

\theoremstyle{definition}

\theoremstyle{remark}

\begin{document}

\title[]{The Hermite-Hadamard Inequality in higher dimensions }
\thanks{This work is supported by the NSF (DMS-1763179) and the Alfred P. Sloan Foundation.}
\author[]{Stefan Steinerberger}
\address{Department of Mathematics, Yale University, New Haven, CT 06511, USA}
\email{stefan.steinerberger@yale.edu}

\begin{abstract} The Hermite-Hadamard inequality states that the average value of a convex function on an interval is bounded from above by the average value of the function at the endpoints of the interval. 
We provide a generalization to higher dimensions: let $\Omega \subset \mathbb{R}^n$ be a convex domain and let $f:\Omega \rightarrow \mathbb{R}$ be a convex function satisfying $f \big|_{\partial \Omega} \geq 0$, then
$$ \frac{1}{|\Omega|} \int_{\Omega}{f ~d \mathcal{H}^n} \leq  \frac{2 \pi^{-1/2} n^{n+1}}{|\partial \Omega|} \int_{\partial \Omega}{f~d \mathcal{H}^{n-1}}.$$
The constant $2 \pi^{-1/2} n^{n+1}$ is presumably far from optimal, however, it cannot be replaced by 1 in general. We prove slightly stronger estimates for the constant in two dimensions where we show that $9/8 \leq c_2 \leq 8$.
We also show, for some universal constant $c>0$, if $\Omega \subset \mathbb{R}^2$ is simply connected with smooth boundary, $f:\Omega \rightarrow \mathbb{R}_{}$ is subharmonic, i.e. $\Delta f \geq 0$, and $f \big|_{\partial \Omega} \geq 0$, then
$$ \int_{\Omega}{f~ d \mathcal{H}^2} \leq c \cdot \mbox{inradius}(\Omega)  \int_{\partial \Omega}{ f ~d\mathcal{H}^{1}}.$$
We also prove that every domain $\Omega \subset \mathbb{R}^n$ whose boundary is 'flat' at a certain scale $\delta$ admits a Hermite-Hadamard inequality for all subharmonic functions
with a constant depending only on the dimension, the measure $|\Omega|$ and the scale $\delta$.
\end{abstract}

\maketitle

\section{Introduction and main results}

\subsection{Introduction.} The Hermite-Hadamard inequality (a 1883 observation of Hermite \cite{hermite} but independently used by Hadamard \cite{hada} in 1893)
for convex $f:[a,b] \rightarrow \mathbb{R}$ states that
$$ \frac{1}{b-a} \int_{a}^{b}{f(x) dx} \leq \frac{f(a) + f(b)}{2}.$$
It is a very elementary consequence of the definition of convexity and has been refined and generalized in innumerable directions (one of the first generalizations is due to Fej\'{e}r \cite{fejer} in 1906).
We
refer to the monograph of Dragomir \& Pearce \cite{drago} collecting many results in that direction. Even though there is a vast abundance of papers on the subject (mathscinet lists over 500 papers
containing 'Hermite-Hadamard' in the title), there seems to be relatively little work outside of the one-dimensional setting. The strongest statement that one could hope for is
$$ \frac{1}{|\Omega|} \int_{\Omega}{f ~d \mathcal{H}^n} \leq \frac{1}{|\partial \Omega|} \int_{\partial \Omega}{f ~d \mathcal{H}^{n-1}},$$
where $\Omega \subset \mathbb{R}^n$ is a convex domain, $f:\Omega \rightarrow \mathbb{R}$ is a convex function that is positive on the boundary $f \big|_{\partial \Omega} \geq 0$ and $\mathcal{H}^k$ is the $k-$dimensional Hausdorff measure (and $| \cdot |$ is the $n-$dimensional and $(n-1)-$dimensional measure, respectively). It was proven for $\Omega = \mathbb{B}_3$ being the $3-$dimensional ball by Dragomir \& Pearce \cite{drago} and $\Omega = \mathbb{B}_n$ by de la Cal \& Carcamo \cite{cal1} (other proofs are given by de la Cal, Carcamo \& Escauriaza \cite{cal2} and Pasteczka \cite{past}). Various other special cases, among them the simplex \cite{bes, now1}, the disk \cite{drago1}, the square \cite{square}, triangles \cite{chen} and Platonic solids \cite{past} have been studied.
 However, as pointed out by Pasteczka \cite{past}, plugging in the affine functions $f(x) = x_i$ for $1\leq i \leq n$ shows that the inequality with constant 1 can only hold if the center of mass of $\Omega$ and $\partial \Omega$ coincide and will fail in general. Moreover, the same example also shows that an assumption like $f_{\partial \Omega} \geq 0$ is necessary: if the centers of mass do not coincide, then there is a linear function for which the two integrals are actually different and if the inequality is strict for $f$ then it fails for $-f$ (which is still linear and thus convex).
 We were attracted to the question because it deals with such a very basic intuition ('a convex function has a higher average on the boundary than on the inside'). We prove several estimates of this flavor, discuss the connection to themes in potential theory and partial differential equations and hope to popularize some of these problems.

\subsection{Results.}
We first state an inequality of Hermite-Hadamard type in higher dimensions; it is a direct analogue of the one-dimensional case and the only one where we obtain an
explicit constant.
\begin{theorem} Let $\Omega \subset \mathbb{R}^n$ be a convex domain and let $f:\Omega \rightarrow \mathbb{R}$ be a convex function such that $f \big|_{\partial \Omega} \geq 0$. Then we have the inequality
$$ \frac{1}{|\Omega|} \int_{\Omega}{f ~d \mathcal{H}^n} \leq \frac{2}{\sqrt{\pi}}\frac{n^{n+1}}{|\partial \Omega|} \int_{\partial \Omega}{f ~d \mathcal{H}^{n-1}}.$$
\end{theorem}
We have not tried to optimize the constant since it is quite clear that our approach will not be able to yield the sharp result. A more precise analysis shows that we can prove
that the optimal constant $c_n$ in $n$ dimensions satisfies
$$ c_n \leq \left(1 + \frac{1}{4n} + o(n^{-1}) \right) n^{n+1}.$$
This is presumably far from optimal. 
What is the best possible constant and for which convex domain $\Omega$ and which function $f$ is it assumed? This may already be interesting in two dimensions where we
establish slightly improved estimates showing that $9/8 \leq c_2 \leq 8$.
An interesting special case  comes from setting $f(x) = x_i$ (which results in a comparison of the
$i-$th coordinate of the center of mass of $\Omega$ and the center of surface mass of $\partial \Omega$). 
\begin{corollary} Let $\Omega \subset \mathbb{R}^n_{ \geq 0}$ be a bounded, convex domain and let 
$$ m_{\Omega} = \frac{1}{| \Omega|} \int_{\Omega}{\textbf{x} ~dx} \qquad \mbox{and} \qquad m_{\partial \Omega} = \frac{1}{|\partial \Omega|} \int_{\partial \Omega}{\textbf{x} ~dx}$$
denote the centers of mass of $\Omega$ and the center of mass of $\partial \Omega$, respectively. Then
$$ \| m_{\Omega} \| \leq \frac{2 n^{n+1}}{\sqrt{\pi}} \| m_{\partial \Omega}\|.$$
\end{corollary}
The inequality in the Corollary is presumably also far from sharp (since it is obtained from assuming that an extremal domain is extremal in each coordinate); what is the sharp form? Can the
extremal domains be characterized?\\
There is a line of reasoning that generalizes Hermite-Hadamard inequalities to domains that are not convex and to functions which are not convex but merely subharmonic, i.e. satisfying
$ \Delta f  \geq 0.$
This idea goes back to Niculescu \& Persson \cite{choquet} (see also \cite{cal2, nicu2}): let $\Omega \subset \mathbb{R}^n$ have a smooth boundary and consider the equation
\begin{align*}
\Delta \phi &= 1 \qquad \mbox{in}~\Omega\\
\phi &= 0 \qquad \mbox{on}~\partial \Omega.
\end{align*}
Then, if  $f$ is subharmonic, i.e. $\Delta f \geq 0$, then we have after an integration by parts
$$ \int_{\Omega}{f~ d \mathcal{H}^n} \leq \int_{\partial \Omega}{ (\nabla \phi \cdot n) f~ d \mathcal{H}^{n-1}},$$
where $n$ is the normal vector on the boundary (pointing outside).
The function $\phi$ arises naturally in the expected lifetime of Brownian motion inside a domain, is of intrinsic interest \cite{makar, payne, payne2, stein} and deeply tied to classical potential theory. We use this connection to prove a fairly general Hermite-Hadamard inequality in two dimensions.

\begin{theorem} There is a universal constant $c>0$ such that for all simply connected domains $\Omega \subset \mathbb{R}^2$ with smooth boundary and all $f:\Omega \rightarrow \mathbb{R}_{}$ satisfying $\Delta f \geq 0$ as well as $f\big|_{\partial \Omega} \geq 0$,
$$ \int_{\Omega}{f~ d \mathcal{H}^2} \leq c \cdot \emph{inradius}(\Omega) \int_{\partial \Omega}{ f ~d\mathcal{H}^{1}}.$$
\end{theorem}
This estimate is easily seen to be sharp up to the constant by taking constant functions on a disk. We observe that taking $f \equiv 1$ implies
$ |\Omega| \lesssim \mbox{inradius}(\Omega) |\partial \Omega|$
for simply connected domains (here and henceforth we use $\lesssim$ and $\gtrsim$ to indicate the presence of a universal constant depending at most on the dimension). Combined with the elementary estimate $ \mbox{inradius}(\Omega) \lesssim |\Omega|^{1/2}$, this implies the isoperimetric inequality $|\partial \Omega| \gtrsim |\Omega|^{1/2}$ with a non-sharp constant.
We do not know what the sharp constant $c$ could be but some heuristic arguments suggest that the range $0.5 \leq c \lesssim 5$ should be a somewhat reasonable guess.
Moreover, the condition that $\Omega$ be simply connected is necessary: consider the annulus
$\Omega = B(0,1) \setminus B(0, \varepsilon)$ and the (harmonic) function $f(x) = -\log{|x|}$ on $\Omega$. We see that, for $\varepsilon$ small,
$$  \int_{\Omega}{f~ d \mathcal{H}^2} \sim  \frac{1}{4} \qquad \mbox{while} \qquad \mbox{inradius}(\Omega) \int_{\partial \Omega}{ f ~d\mathcal{H}^{1}} \sim \pi \varepsilon \log{\frac{1}{\varepsilon}}.$$
Taking $f \equiv -1$ shows that some restriction on the sign of the function is necessary, our proof shows that $f \geq 0$ on the boundary is sufficient to imply the result.
Theorem 2 can be interpreted as a refinement of the
maximum principle: the maximum principle states that the maximum value of a subharmonic function is assumed on the boundary. It also implies that if $f$ is of a certain size in the interior of
$\Omega$, then it has to be at least of comparable size on a nontrivial portion of the boundary. We also show such a result for convex domains in $n \geq 3$ dimensions.

\begin{theorem} Let $\Omega \subset \mathbb{R}^n$, $n \geq 3$, be a convex domain and let $f:\Omega \rightarrow \mathbb{R}_{}$ satisfy $\Delta f \geq 0$ as well as $f\big|_{\partial \Omega} \geq 0$. Then, for a constant $c_n$ depending only on the dimension, we have that
$$  \int_{\Omega}{f ~d \mathcal{H}^n} \leq c_n|\Omega|^{1/n}  \int_{\partial \Omega}{f ~d \mathcal{H}^{n-1}}.$$
\end{theorem}

We conclude with a rather general result that shows that Hermite-Hadamard inequalities are always possible on domains whose boundary is sufficiently flat at a certain scale.
We will say that
$\Omega$ has a boundary $\partial \Omega$ that is 'flat at scale $\delta$' if for every point $y \in \partial \Omega$ the set $B(y, \delta) \cap \partial \Omega$ can, possibly after a rotation, be locally
written as the graph of a differentiable function $\psi$ with derivative $| \nabla \psi| \leq 1/10$. There are certainly various other conditions of a similar flavor that could be imposed to obtain similar
results.

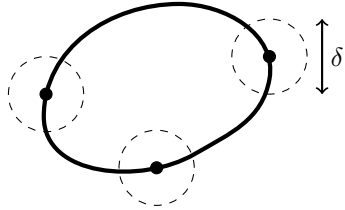
\begin{figure}[h!]
\begin{center}
\begin{tikzpicture}
\draw [ultra thick] (0,0) to[out=30, in=260] (1,1) to[out=80, in =0] (-0.2, 2) to[out=180, in =90] (-2, 0.5) to[out=270, in =210] (0,0);
\filldraw (-1.97, 0.8) circle (0.08cm);
\draw [dashed] (-1.97, 0.8) circle (0.5cm);
\filldraw (1, 1.3) circle (0.08cm);
\draw [dashed] (1, 1.3) circle (0.5cm);
\filldraw (-0.5, -0.18) circle (0.08cm);
\draw [dashed] (-0.5, -0.18) circle (0.5cm);
\draw [thick, <->] (1.7, 0.8) -- (1.7, 1.8);
\node at (1.9, 1.3) {$\delta$};
\end{tikzpicture}
\end{center}
\caption{Flatness at scale $\delta$: the boundary restricted to a $\delta-$ball centered at a point on the boundary is the graph of a $C^1$ function with small derivative.}
\end{figure}

\begin{theorem} Let $\Omega \subset \mathbb{R}^n$ have a smooth boundary that is flat at scale $\delta$. There exists a constant $c_n >0$ depending only on the dimension such that for all subharmonic $f:\Omega \rightarrow \mathbb{R}_{\geq 0}$ 
$$  \int_{\Omega}{f ~d \mathcal{H}^n} \lesssim_n   \delta^{n-1} \exp\left( \frac{ c_n |\Omega|^{2/n}}{\delta^2} \right)   |\Omega|^{-\frac{n-2}{n}}  \int_{\partial \Omega}{f ~d \mathcal{H}^{n-1}}.$$
\end{theorem}
The estimate seems far from optimal but is not too far from the truth in some cases. If $B_R \subset \mathbb{R}^n$ is a ball of radius $R$, then $\delta \sim R$ and 
$$    \delta^{n-1} \exp\left( \frac{ c |\Omega|^{2/n}}{\delta^2} \right)   |\Omega|^{-\frac{n-2}{n}}   \sim R^{n-1} \exp\left( \frac{ cR^2}{R^2} \right) R^{2-n} \sim R$$
which is easily seen (by taking constant functions) to be optimal. Moreover, if we rescale a domain $\Omega$ by a factor $\lambda$, then $\delta$ is rescaled by the same factor and we see that the
bound scales also like $\lambda$ (thus respecting the symmetry under dilations).

\section{Proof of Theorem 1}
\subsection{Outline} Suppose $\phi:\Omega \rightarrow \mathbb{S}^{n-1}$ is a continuous map. For every point $x \in \Omega$, we can consider the intersection of the line $x + t \phi(x)$ with $\partial \Omega$. If $x$ is strictly inside the convex domain, then there are exactly two intersections with the boundary in $y_1, y_2 \in \partial \Omega$. There exists a $0 < t < 1$  
(the unique one solving $t y_1 + (1-t) y_2 = x$) such that
$$ f(x) \leq t f(y_1) + (1-t) f(y_2).$$
Our proof is based on interpreting this geometric fact as a mapping of the point $x$ to two (weighted) points $y_1, y_2$ on the boundary (the weights being $t$ and $1-t$). By integrating in a tiny neighborhood of $x$, we can interpret it as a way of transporting Lebesgue measure to the boundary via the mapping $\phi$. The main idea of our proof is now encapsulated in the following Lemma.

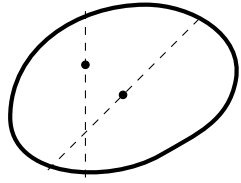
\begin{figure}[h!]
\begin{center}
\begin{tikzpicture}
\draw [ultra thick] (0,0) to[out=30, in=260] (1,1) to[out=80, in =0] (-0.2, 2) to[out=180, in =90] (-2, 0.5) to[out=270, in =210] (0,0);
\filldraw (-0.5, 0.8) circle (0.05cm);
\draw [dashed] (-1.5, -0.2) -- (0.5,1.8);
\filldraw (-1, 1.2) circle (0.05cm);
\draw [dashed] (-1, 1.2-1.5) -- (-1, 1.2+0.8);
\end{tikzpicture}
\end{center}
\caption{Associating a direction to every point resulting in two points on the boundary whose weights are determined by the location of the point in the interior.}
\end{figure}

\begin{lemma} If this particular push-forward of the normalized Lebesgue measure under $\phi$ gives rise to an absolutely continuous measure on the boundary whose Radon-Nikodym derivative with respect to $\mathcal{H}^{n-1}$ is bounded from above by $c$, then for all convex functions $f:\Omega \rightarrow \mathbb{R}$ for which $f \big|_{\partial \Omega} \geq 0$
$$ \int_{\Omega}{f ~d \mathcal{H}^n} \leq c \int_{\partial \Omega}{f~d \mathcal{H}^{n-1}}.$$
\end{lemma}
This simple statement is actually intuitively used in a variety of settings (it is the natural statement one encounters when proving the Hermite-Hadamard inequality on $\Omega=[0,1]^2$), one could also reasonably re-interpret Pasteczka's proof for $\Omega = \mathbb{B}_n$ \cite{past} in this light.\\

This Lemma actually highlights a rather curious perspective on the problem: a natural way to construct upper bounds is to construct a map $\phi:\Omega \rightarrow \mathbb{S}^{n-1}$ such that the distribution of the push-forward of the Lebesgue measure is as flat as possible with respect to the boundary measure. 
In particular, if there exists a map $\phi$ moving the Lebesgue measure exactly to Hausdorff measure $\mathcal{H}^{n-1}$ on the boundary, then the Hermite-Hadamard inequality holds with
constant 1. This seems to connect the problem to question in optimal transport \cite{vi}.
 We also see a connection to Choquet theory (see \cite{nicu}) but only in a vague sense: we are not at all interested in having the measure on the boundary concentrated in the set of extreme points
but want it to be spread out everywhere. Nonetheless, the connection to Choquet theory may potentially be of interest for further developments. Another interesting connection, as already pointed out in \cite{cal2}
is the notion of convex ordering of random variables \cite{mu, sh}.

\subsection{Proof of Lemma 1}
\begin{proof} The proof is fairly straight-forward: we subdivide $\Omega$ into cubes of size $\varepsilon$ (which will later go to 0). We use the continuous map $\phi$ to define a line in $\Omega$ going through $x$ via
 $x + t \phi(x)$. As discussed above, if $x$ is strictly inside the convex domain, then there are exactly two intersections with the boundary in $y_1, y_2 \in \partial \Omega$ because of the convexity of $\Omega$. There exists a $0 < t < 1$  
(the unique one solving $t y_1 + (1-t) y_2 = x$) such that
$$ f(x) \leq t f(y_1) + (1-t) f(y_2).$$
We may interpret this as the integral of the Lebesgue measure of the cube being dominated by the push-forward of the measure on the boundary against the integral on the boundary. This can be carried out in every point and results in the following inequality: if the push-forward of the Lebesgue measure results in an measure $\mu$ on the boundary, then for all convex functions $f$ 
$$ \int_{\Omega}{f  d \mathcal{H}^n} \leq \int_{\partial \Omega}{f d\mu}.$$
Since $\mu$ is absolutely continuous with respect to the surface measure, we may write this as
$$ \int_{\partial \Omega}{f d\mu} = \int_{\partial \Omega}{f \left[ \frac{d \mu}{d \mathcal{H}^{n-1}} \right] d \mathcal{H}^{n-1}},$$
where the expression in the parentheses is the Radon-Nikodym derivative with respect to the $(n-1)-$dimensional surface measure.
Since $f$ is nonnegative on the boundary, we can estimate this from above by 
$$  \int_{\partial \Omega}{f d \mu} \leq \max_{x \in \partial \Omega}{ \left[ \frac{d \mu}{d \mathcal{H}^{n-1}} \right]}   \int_{\partial \Omega}{f d \mathcal{H}^{n-1}} = c \int_{\partial \Omega}{f d \mathcal{H}^{n-1}}$$
which is the desired statement.
\end{proof}

\subsection{Proof of Theorem 1.}
Our main ingredients besides Lemma 1 are (1) the John ellipsoid theorem
stating that for every convex domain $\Omega \subset \mathbb{R}^n$ there exists an ellipsoid $E$ such that
$$ E \subseteq \Omega \subseteq n E$$
and (2) a classical formula of Cauchy for the surface area of a convex body: for any unit vector
$v \in \mathbb{R}^n$, we use $\pi_v:\mathbb{R}^{n} \rightarrow \mathbb{R}^{n-1}$ to denote the projection onto the hyperplane orthogonal to $v$. Cauchy's formula then states that the surface
area of the boundary is given by
$$ \mathcal{H}^{n-1}( \partial \Omega) = \frac{1}{|\mathbb{B}^{n-1}|} \int_{\mathbb{S}^{n-1}}{ \mathcal{H}^{n-1}(\pi_v \Omega) dv}.$$
The use of Cauchy's formula could be bypassed since we are only using it to estimate the surface area of an ellipsoid which could conceivably be done by
various other means (some of which might end up resulting in slightly improved constants; however, we do not see how any approach of this flavor could possibly yield a sharp result). Cauchy's formula results in a particularly simple algebraic expression which is why we favor this approach for clarity of exposition.

\begin{proof}
We use the John ellipsoid theorem: there exists an Ellipsoid $E$ such that
$$ E \subseteq \Omega \subseteq n E.$$
Let us furthermore assume, after possibly rotating and translating both $E$ and the convex body, that the boundary of $E$ is described by the equation
$$ \frac{x_1^2}{a_1^2} + \frac{x_2^2}{a_2^2} + \dots + \frac{x_n^2}{a_n^2} = 1$$ 
where, without loss of generality, $ a_1 > a_2 > a_3 > \dots > a_n$.
 The volume of the ellipsoid is given by a scaling of the unit ball and thus 
$$ |E| = \frac{\pi^{n/2}}{\Gamma\left(\frac{n}{2} + 1\right)} \prod_{i=1}^{n}{a_i}.$$
We now use Lemma 1 and define $\phi \equiv e_n \chi_{x_n > 0}$ where $e_n$ is the unit vector $e_n = (0,0,\dots, 0,1)$. Then, by construction, the Radon-Nikodym derivative of the push-forward of the Lebesgue measure
is bounded from above by $n a_n$ for every point on the boundary $\partial \Omega$. Since we are working with normalized Lebesgue measure, we see that this allows us to use the
Lemma with 
$$ c = \frac{ n a_n}{|\Omega|} \leq \frac{ n a_n}{|E|} = n  \Gamma\left(\frac{n}{2} + 1\right) \pi^{-n/2}   \left( \prod_{i=1}^{n-1}{a_i} \right)^{-1}.  \qquad (\diamond)$$
If one convex set contains another, then the $(n-1)-$dimensional Hausdorff measure of their boundaries obey the same ordering (this follows neatly from the Cauchy formula) and thus
$$ \mathcal{H}^{n-1}(\partial \Omega) \leq \mathcal{H}^{n-1}(\partial n E)  = n^{n-1} \mathcal{H}^{n-1}(\partial E).$$
It remains to estimate the surface area of the ellipsoid from above. Cauchy's formula yields
$$ \mathcal{H}^{n-1}(\partial E) = \frac{1}{|\mathbb{B}^{n-1}|} \int_{\mathbb{S}^{n-1}}{ \mathcal{H}^{n-1}(\pi_v E) dv} \leq \frac{\mathcal{H}^{n-1}(\mathbb{S}^{n-1})}{\mathcal{H}^{n-1}(\mathbb{B}^{n-1})} \max_{v \in \mathbb{S}^{n-1}(E)} \mathcal{H}^{n-1}(\pi_v E).$$
A simple computation shows that
$$ \mathcal{H}^{n-1}(\mathbb{B}^{n-1}) =  \frac{\pi^{(n-1)/2}}{\Gamma\left(\frac{n+1}{2}\right)} \quad \mbox{and} \quad  \mathcal{H}^{n-1}(\mathbb{S}^{n-1}) =  \frac{2 \pi^{n/2}}{\Gamma(\frac{n}{2})}$$
which we estimate from above as 
$$ \frac{\mathcal{H}^{n-1}(\mathbb{S}^{n-1})}{\mathcal{H}^{n-1}(\mathbb{B}^{n-1})} = 2 \sqrt{\pi} \frac{\Gamma\left(\frac{n+1}{2}\right)}{\Gamma\left(\frac{n}{2}\right)} \leq \sqrt{2\pi n}.$$
and therefore
$$ \mathcal{H}^{n-1}(\partial E) \leq \sqrt{2\pi n} \max_{v \in \mathbb{S}^{n-1}(E)} \mathcal{H}^{n-1}(\pi_v E).$$
It remains to compute the maximal $(n-1)-$dimensional volume of a projection. One is naturally inclined to believe that that volume is maximized under projection in direction of the shortest axis: this is
correct, we refer to Connelly \& Ostro \cite{con} as well as Rivin \cite{rivin}. Therefore
$$ \mathcal{H}^{n-1}(\partial E) \leq \sqrt{2 \pi n} \frac{\pi^{\frac{n-1}{2}}}{\Gamma\left(\frac{n-1}{2} + 1\right)} \prod_{i=1}^{n-1}{a_i}.$$
Altogether, using $(\diamond)$ and simplifying the arising the expressions, we obtain a fairly simply estimate on the desired value
$$c \cdot \mathcal{H}^{n-1}(\partial \Omega) \leq    \sqrt{2n}  \frac{\Gamma\left(\frac{n+2}{2}\right)}{\Gamma\left(\frac{n+1}{2}\right)} n^{n}.$$
This expression can be estimated by
$$  \sqrt{2n}  \frac{\Gamma\left(\frac{n+2}{2}\right)}{\Gamma\left(\frac{n+1}{2}\right)} n^{n} \leq \frac{2  n^{n+1}}{\sqrt{\pi}} \qquad \mbox{with equality for}~n=2.$$
Moreover, we have the asymptotic behavior
$$   \sqrt{2n}  \frac{\Gamma\left(\frac{n+2}{2}\right)}{\Gamma\left(\frac{n+1}{2}\right)} n^{n} = \left(1 + \frac{1}{4n} + o(n^{-1}) \right) n^{n+1}.$$
\end{proof}

The proof yields a better constant if the domain $\Omega$ is symmetric, i.e. if $x \in \Omega$ implies $-x \in \Omega$: for symmetric
domains there is an improved constant in John's theorem stating
$$ E \subseteq \Omega \subseteq \sqrt{n} E.$$

\subsection{Explicit bounds for $n=2$.} This problem of finding precise constants seems to already be difficult for $n=2$ dimensions. We record the following slight improvement over the constant $c_2 \leq 9.02\dots$ that is given by Theorem 1 and derive a lower bound.

\begin{proposition}
There exists a universal constant 
$$ \frac{9}{8} < c_2 \leq 8$$
 such that if $\Omega \subset \mathbb{R}^2$ is a convex domain and $f:\Omega \rightarrow \mathbb{R}$ is a convex function that is positive on the boundary $f \big|_{\partial \Omega} \geq 0$, then
$$ \frac{1}{|\Omega|} \int_{\Omega}{f ~d \mathcal{H}^2} \leq \frac{c_2}{|\partial \Omega|} \int_{\partial \Omega}{f ~d \mathcal{H}^{1}}.$$
\end{proposition}
\begin{proof} The proof is more or less the same as the proof of Theorem 1 but we pay more attention to the constants. We assume that the John ellipsoid 
$$A \subseteq \Omega \subseteq 2 A$$ 
has semi-axes
$a$ and $b$ where $a \geq b$. Then we can apply Lemma 1 to conclude that
$$ \int_{\Omega}{f ~d \mathcal{H}^n} \leq 2b \int_{\partial \Omega}{f~d \mathcal{H}^{n-1}}.$$
This shows that the constant we are trying to bound is given by
$$ \frac{2b |\partial \Omega|}{|\Omega|} \leq \frac{4b |\partial A|}{|A|} = \frac{4b |\partial A|}{\pi a b} =  \frac{4 |\partial A|}{\pi a}.$$
The circumference of an ellipsoid is given by
$$ |\partial A| = 4 a \int_{0}^{\pi/2}{\sqrt{1- e^2 \sin{\theta}} d\theta} =  4a E\left(\sqrt{1- \frac{b^2}{a^2}}\right),$$
where $e$ is the eccentricity $\sqrt{1-b^2/a^2}$ and $E$ is the complete elliptic integral of the second kind which is monotonically decreasing.  Therefore
$ |\partial A| \leq 2 \pi a$ and thus
$$ \frac{2b |\partial \Omega|}{|\Omega|} \leq 8.$$  As for the lower bound, we consider the domain
$$ \Omega = \left\{(x,y) \in \mathbb{R}^2: 0 \leq x \leq 1 \wedge 0 \leq y \leq a x\right\}$$
and the function $f(x,y) = x$. We observe that
$$ \frac{1}{|\Omega|} \int_{\Omega}{x dx} = \frac{1}{a} \int_{0}^{a}{x 2a x dx} = \frac{2}{3}.$$
Moreover,
$$ \int_{\partial \Omega}{x dx} = \sqrt{1 + a^2}  + 2a \qquad \mbox{and} \qquad |\partial \Omega| = 2 \sqrt{1 + a^2} + 2a.$$
Altogether, this implies
$$ c_2 \geq \sup_{a > 0} \frac{3}{2}\frac{\sqrt{1 + a^2}  + 2a}{ 2 \sqrt{1 + a^2} + 2a} = \frac{9}{8}.$$
\end{proof}

We remark that the proof actually shows slightly more and we obtain the bound
$$ c_{\Omega} \leq  \frac{16}{\pi}  E\left(\sqrt{1- \frac{b^2}{a^2}}\right) \qquad \mbox{for convex domains}~\Omega \subset \mathbb{R}^2$$
whose John ellipsoid has semi-axes $a$ and $b$.

\section{Proof of Theorem 2}
\subsection{Green's function.}
 Let $\Omega \subset \mathbb{R}^n$ be sufficiently regular so that the Dirichlet problem
\begin{align*}
\Delta u &= 0 \qquad \mbox{in}~\Omega \\
u &= g \qquad \mbox{on}~\partial \Omega
\end{align*}
has a solution given by a Green's function $p:\Omega \times \partial \Omega \rightarrow \mathbb{R}_{\geq 0}$ via
$$ u(x) = \int_{\partial \Omega}{p(x,y) g(y) dy}.$$
Plugging in constant functions for $u$ shows that 
$$ \int_{\partial \Omega}{p(x,y) dy} = 1.$$
There exists a natural bound on the constant in the Hermite-Hadamard inequality for subharmonic functions that is determined by integrating out the other variable (this observation
is essentially contained in \cite{cal2} but phrased in a somewhat different language).
\begin{thm}[de la Cal, Carcamo \& Escauriaza \cite{cal2}] If $f \in C^2(\Omega)$ is subharmonic, then
$$ \int_{\Omega}{f~d \mathcal{H}^n} \leq \int_{\partial \Omega}{\left( \int_{\Omega}{p(x, y) dx} \right)f(y)~d\mathcal{H}^{n-1}}$$
with equality if and only if $f$ is harmonic.
\end{thm}
We will use this bound to establish Theorem 2 and Theorem 3. Finally, we prove a generalization of this inequality in Section \S 5 and use it to prove Theorem 4.

\subsection{Proof of Theorem 2}

\begin{proof} The argument is a consequence of the inequality of de la Cal, Carcamo \& Escauriaza and a rigidity statement for simply connected domains in the plane: if $\Omega \subset \mathbb{R}^2$ is simply connected with a smooth boundary, then
$$ \max_{y \in \partial \Omega} \int_{\Omega}{p(x, y) dx} \leq c \cdot \mbox{inradius}(\Omega)$$
for some universal constant $c$.
This inequality, while possibly not being stated anywhere in particular, follows from a standard rigidity phenomenon in potential theory that is encountered in various other problems as well. We refer to Banuelos \cite{ban}  who lists four particular instances of the phenomenon and the associated bounds as well as to a paper of Rachh and the author \cite{manas} for implications for the Schr\"odinger equation. For a formal argument, it suffices to note that simply connected domains in $\mathbb{R}^2$ have the property that any $\delta-$disk around a point $x$ on the boundary satisfies
$$ \mathcal{H}^1(B(x, \delta) \cap \partial \Omega) \gtrsim \delta \qquad \mbox{for}~0 < \delta \leq \mbox{diam}(\Omega).$$
\begin{center}
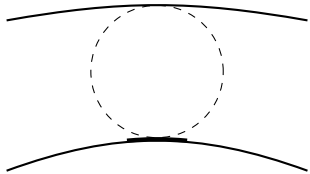
\begin{figure}[h!]
\begin{tikzpicture}[scale=2]
\draw [thick] (0,0) to[out=20, in=160] (2,0);
\draw [ultra thick] (0.8,0.2) to[out=5, in=175] (1.2,0.2);
\draw [thick] (0,1) to[out=10, in=170] (2,1);
\draw [dashed]  (1, 0.66) circle (0.44cm);
\end{tikzpicture}
\caption{Simply connected domains in $\mathbb{R}^2$ have the property that for every point on the boundary, the total amount of boundary in a $\delta-$disk is $\gtrsim \delta$ for all $\delta$ up to $\mbox{diam}(\Omega)$.}
\end{figure}
\end{center}

For a segment to be particularly accessible, we require that there is not a lot of other boundary nearby from which the result then follows by scaling. 
\end{proof}

It could be of interest to relate the constant $c$ to the other constants attached to other incarnations of the same phenomenon (see \cite{ban}). We do not know what an extremizing domain could look like. A natural candidate is given in Fig. 3 (the largest density being assumed at the tip of the slit). 
The same rigidity phenomenon should also governs the gradient of the torsion function. More precisely, it could be of interest to determine whether for a simply connected domain $\Omega \subset \mathbb{R}^2$ the solution of 
$$
\Delta \phi = 1~ \mbox{in}~\Omega \qquad \mbox{and} \qquad \phi = 0 ~ \mbox{on}~\partial \Omega
$$
satisfies $\|\nabla u\|_{L^{\infty}} \leq c \cdot \mbox{inradius}(\Omega)$ and to find the domain for which the constant is extremal. This would allow for another proof of Theorem 2 via the inequality of Niculescu \& Persson \cite{choquet}.
 We emphasize that these types of questions tend to come in two parts: existence of a constant (which is usually not so difficult) and establishing the sharp value (which, especially for this problems of this type, is notoriously hard and often tied to some very elusive constants in classical complex analysis, see \cite{ban1, ban2, car1}).

\begin{center}
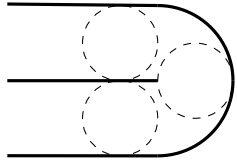
\begin{figure}[h!]
\begin{tikzpicture}[scale=2]
\draw [very thick] (1,0.001) to[out=0, in=180] (2,0);
   \draw [very thick,domain=270:360] plot ({2+0.5*cos(\x)}, {0.5+0.5*sin(\x)});
   \draw [very thick,domain=0:90] plot ({2+0.5*cos(\x)}, {0.5+0.5*sin(\x)});
\draw [very thick] (2,1) -- (1,1.01);
\draw [very thick] (1,0.5) -- (2,0.5);
\draw [dashed] (2.25, 0.5) circle (0.25cm);
\draw [dashed] (1.75, 0.75) circle (0.25cm);
\draw [dashed] (1.75, 0.25) circle (0.25cm);
\end{tikzpicture}
\caption{A possible candidate for an extremizing domain.}
\end{figure}
\end{center}

\section{Proof of Theorem 3}

\begin{proof} The argument combines various ideas and estimates developed in the arguments above. In particular, we will make use of the estimate, valid for subharmonic functions satisfying $f \big|_{\partial \Omega} \geq 0$,

$$ \int_{\Omega}{f~d \mathcal{H}^n} \leq \int_{\partial \Omega}{\left( \int_{\Omega}{p(x, y) dx} \right)f(y)~d\mathcal{H}^{n-1}} \leq \left( \max_{y \in \partial \Omega}{  \int_{\Omega}{p(x, y) dx}} \right)\int_{\partial \Omega}{f~d\mathcal{H}^{n-1}} $$
and will establish the desired result by showing that for convex $\Omega \subset \mathbb{R}^n$
$$ \max_{y \in \partial \Omega}{  \int_{\Omega}{p(x, y) dx}} \lesssim |\Omega|^{\frac{1}{n}}.$$
After fixing an arbitrary point $y \in \partial \Omega$, we translate and rotate $\Omega$ in such a way that $y = 0$ w.l.o.g. and the entire domain $\Omega$ is contained in $\left\{x_n \geq 0 \right\}$.
By domain monotonicity, we can bound the Green's function associated to $\Omega$ by
$$ p_{\Omega}(x,0) \leq p_{\mathbb{R}^{n-1} \times \mathbb{R}_{\geq 0}}(x,0) = \Gamma\left(\frac{n}{2}\right) \frac{1}{\pi^{n/2}} \frac{x_n}{\|x\|^n}$$
which is simply the classical Poisson kernel on the upper half space. 

\begin{center}
\begin{figure}[h!]
\begin{tikzpicture}
\draw [very thick] [rotate=45] (0,0) ellipse (2cm and 1cm);
\filldraw (-0.85,-1.57) circle (0.06cm);
\draw [<->, thick] (-3, -1.57) -- (2, -1.57);
\node at (-0.9, -1.8) {$y = 0$};
\node at (2, -1.8) {$\mathbb{R}^{n-1}$};
\draw [thick, ->] (-0.85, -1.57) -- (-0.85, 2);
\node at (-0.6, 1.8) {$\mathbb{R}$};
\draw [dashed] (-3, 0) -- (2,0);
\draw [dashed] (-3, 1) -- (2,1);
\draw [dashed] (-3, -1) -- (2,-1);
\draw [<->] (4,-1.57) -- (4,2);
\node at (4.4, 0) {$a$};
\end{tikzpicture}
\caption{The domain $\Omega$ before Schwarz symmetrization.}
\end{figure}
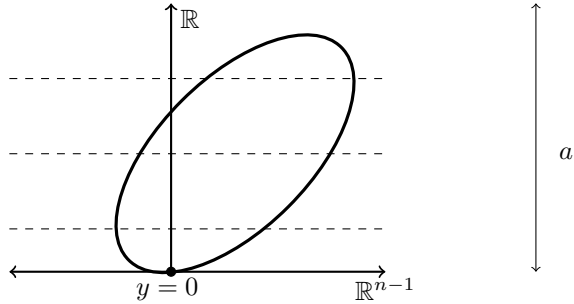
\end{center}

We write $x_{n-1} = (x_1, \dots, x_{n-1}, 0)$ and estimate
$$  \Gamma\left(\frac{n}{2}\right) \frac{1}{\pi^{n/2}} \int_{\Omega}{p(x, 0) dx} = \Gamma\left(\frac{n}{2}\right) \frac{1}{\pi^{n/2}}  \int_{\Omega}{ \frac{x_n}{(x_n^2 + \|x_{n-1}\|^2)^{n/2}} d x_{n-1} dx_{n}}.$$ 
We use Fubini's theorem and estimate the integral on slices $\left\{x_n = \mbox{const}\right\}$. The classical rearrangement inequality implies that
$$ \mbox{the expression} \qquad \int_{\Omega \cap \left\{x_n ={\tiny \mbox{const}}\right\}}{ \frac{x_n}{(x_n^2 + \|x_{n-1}\|^2)^{n/2}} d x_{n-1}} \qquad \mbox{is maximized}$$
 if $\Omega$ is an $(n-1)-$dimensional ball centered around $(0_{n-1}, \mbox{const})$ having the same $(n-1)-$dimensional volume as $\Omega \cap \left\{x_n = \mbox{const}\right\}$. 
We can thus switch to polar coordinates and bound
$$ \int_{\Omega}{p(x, 0) dx} \lesssim \int_{0}^{a}{ \int_{0}^{ c_n \left| \Omega \cap \left\{x_n = {\tiny \mbox{const}}\right\} \right|^{1/(n-1)}}{ \frac{x_n \cdot \|x_{n-1}\|^{n-2}}{(x_n^2 + \|x_{n-1}\|^2)^{n/2}}  d x_{n-1} dx_{n}}},$$
where $c_n$ depends only on the dimension.
We exchange the order of integration, let $a \rightarrow \infty$ and use that, for all $c>0$,
$$ \int_{0}^{\infty}{ \frac{ b c^{n-2}}{(b^2+ c^2)^{n/2}} db} = \frac{1}{n-2}$$
to bound
$$ \int_{\Omega}{p(x, 0) dx} \lesssim  \sup_{c > 0} \left| \Omega \cap \left\{x_n = c\right\} \right|^{\frac{1}{n-1}}.$$
Let us now use $\Omega^*$ to denote the domain obtained from the rearrangement process described above where $(n-1)-$dimensional slices are being replaced by $(n-1)-$dimensional balls centered around the $x_n-$axis. This is merely the Schwarz symmetrization and is known to preserves convexity (this is a consequence of the Brunn-Minkowski inequality, see Berger \cite{berger}). In particular, $\Omega^*$ is convex, it is radial around $x_n$ and has height $a$. Moreover, by construction, for all $c>0$
$$  \left| \Omega \cap \left\{x_n = c\right\} \right|^{\frac{1}{n-1}} = \left| \Omega^* \cap \left\{x_n = c\right\} \right|^{\frac{1}{n-1}}.$$
This shows that $\Omega^*$ is contained in the cylinder
$$ \Omega^* \subset \left\{ (x_{n-1}, x): \|x_{n-1}\| \leq  c_n\left(\sup_{c > 0} \left| \Omega \cap \left\{x_n = c\right\} \right|\right)^{\frac{1}{n-1}} ~\mbox{and}~ 0 \leq x_n \leq a \right\},$$
where $c_n$ is yet another constant depending only on the dimension. 
In particular, $\Omega^*$ does not fit into any smaller cylinder. 
By construction and convexity of $\Omega^*$, this shows that
$$ |\Omega| = |\Omega^*| \gtrsim_n  \left(\sup_{c  > 0} \left| \Omega \cap \left\{x_n = c\right\} \right|\right)^{}  a,$$
where the implicit constant depends only on the dimension. We now prove, somewhat independently, the bound
$$ \int_{\Omega}{p(x, 0) dx} \leq a.$$
This is fairly easy to see: we bound the integral from above by replacing the domain $\Omega$ by the much bigger domain 
$ \Omega \subset \left\{(x_{n-1}, x_n): 0 \leq x_n \leq a\right\}.$ Translational invariance (or direct computation) show that the integral simplifies exactly to $a$. Altogether, this means that we have two estimates
that can be combined in an interpolatory fashion
\begin{align*}
  \int_{\Omega}{p(x, 0) dx} &\lesssim \min\left\{  \sup_{c > 0} \left| \Omega^* \cap \left\{x_n = c\right\} \right|^{\frac{1}{n-1}}, a \right\} \\
&\leq \left( \sup_{c > 0} \left| \Omega^* \cap \left\{x_n = c\right\} \right|^{\frac{1}{n-1}} \right)^{\frac{n-1}{n}} a^{\frac{1}{n}} \lesssim |\Omega|^{\frac{1}{n}}.
\end{align*}
\end{proof}

\section{Proof of Theorem 4}
\subsection{Idea.} The idea behind the proof of Theorem 4 is to make use of the implicit monotonicity formula that underlies the inequality of de la Cal, Carcamo \& Escauriaza \cite{cal2} and exploit time as
an additional parameter. More precisely, let $\Omega \subset \mathbb{R}^n$ be a domain with smooth boundary and let $f:\Omega \rightarrow \mathbb{R}$ be a subharmonic function.
The heat equation
\begin{align*}
 \frac{\partial}{\partial t}u(t,x) - \Delta u(t,x) &= 0 ~ \mbox{in}~\Omega\\
u(t, x) &= f(x)~\mbox{for}~x \in \partial \Omega\\
u(0,x) &= f(x).
\end{align*}
has the property that
$$ \frac{\partial}{\partial t} \int_{\Omega}{ u(t,x) dx} \geq 0.$$
Moreover, for any $x \in \Omega$, we have
$$ \lim_{t \rightarrow \infty}{ u(t,x) } = \int_{\partial \Omega}{p(x, y)f(y)~d\mathcal{H}^{n-1}}.$$
We will now use the same principle for a finite $0 < t < \infty$ to deduce Hermite-Hadamard inequalities.
 By standard theory, we know that we can write the solution of the heat equation as
$$ u(t,x) = \int_{\Omega}{ p_t(x,y)f(y) dy} + \int_{\partial \Omega}{ q_t(x,z) f(z) dz},$$
where, for all $x \in \Omega$ and all $t > 0$
$$ \int_{\Omega}{ p_t(x,y) dy} + \int_{\partial \Omega}{ q_t(x,z) dz} = 1.$$
The main idea is as follows: if we can find a time $t>0$ such that
$$ \max_{x \in \Omega}\int_{\Omega}{ p_t(x,y) dy} \leq 1 - \delta \qquad \mbox{as well as} \qquad  \max_{z \in \partial \Omega}\int_{\Omega}{q_t(x,z) dx} \leq \eta,$$
then we can use the identity $p_t(x,y) = p_t(y,x)$ to estimate
\begin{align*}
\int_{\Omega}{ f(x) dx} &\leq \int_{\Omega}{ u(t,x) dx}   = \int_{\Omega}{  \int_{\Omega}{ p_t(x,y)f(y) dy} dx} +  \int_{\Omega}{\int_{\partial \Omega}{ q_t(x,z) f(z) dz}dx}  \\
&= \int_{\Omega}{  \int_{\Omega}{ p_t(y,x)f(y) dy} dx} +  \int_{\Omega}{\int_{\partial \Omega}{ q_t(x,z) f(z) dz}dx}  \\
&\leq \left(1 -  \delta \right)\int_{\Omega}{ f(x) dx} +  \eta\int_{\partial \Omega}{  f(z) dz}
\end{align*}
which implies a Hermite-Hadamard inequality with constant $\eta \delta^{-1}$.

\subsection{Proof of Theorem 4}
\begin{proof}
We start by showing a fairly standard isoperimetric estimate (that is, for example, also used in \cite{alex}): for every $t > 0$
$$ \max_{x \in \Omega}  \int_{\Omega}{ p_t(x,y) dy} \leq 1 - c_n^{-1} |\Omega|^{\frac{n-2}{n}} t^{1 - \frac{n}{2}} e^{-c_n^2 |\Omega|^{2/n}t^{-1}},$$
where the constant $c_n$ depends only on the dimension (and may change its value from line to line). This is done via a probabilistic interpretation: the integral
$$  \int_{\Omega}{ p_t(x,y) dy} \qquad \mbox{has an interpretation}$$
as the likelihood of a Brownian motion started in $x$ never touching the boundary $\partial \Omega$ for all points in time $0 < t^* < t$.
We can bound the likelihood of this event from above by bounding the likelihood of its negation from below.
The likelihood of Brownian motion being outside $\Omega$ at some point $0 < t^* < t$ is certainly larger than the likelihood of it being outside at time $t$ (because that
would imply it touching the boundary at some intermediate point). Then, however, the classical rearragement inequality implies that
this likelihood is minimized by having $\Omega$ be a ball centered around $x$. That likelihood can be written as
$$ \int_{\mathbb{R}^n \setminus B(0, c_n|\Omega|^{1/n})}{ \frac{1}{(4 \pi t)^{n/2} } e^{- \frac{\|x\|^2}{4t}} dx} =
\frac{1}{(4 \pi t)^{n/2}} \int_{c_n |\Omega|^{1/n}}^{\infty}{ e^{-\frac{r^2}{4t}} r^{n-1} dr}.
$$
This integral is fairly easy to bound from below since $n \geq 2$ and thus
\begin{align*}
\frac{1}{(4 \pi t)^{n/2}} \int_{c_n |\Omega|^{1/n}}^{\infty}{ e^{-\frac{r^2}{4t}} r^{n-1} dr} &\geq \frac{c_n^{n-2} |\Omega|^{\frac{n-2}{n}}}{(4 \pi t)^{n/2}} \int_{c_n |\Omega|^{1/n}}^{\infty}{ e^{-\frac{r^2}{4t}} r^{} dr}\\
&= \frac{c_n^{n-2} |\Omega|^{\frac{n-2}{n}}}{(4 \pi t)^{n/2}} 2t e^{-\frac{c_n^2 |\Omega|^{2/n}}{4t}} \\
&\gtrsim_n  |\Omega|^{\frac{n-2}{n}} t^{1 - \frac{n}{2}} e^{-c_n^2 |\Omega|^{2/n}t^{-1}}.
\end{align*}

At the same time, if $z \in \partial \Omega$ and $\partial \Omega$ is flat around $z$ at scale $\sqrt{t}$, then the standard bounds on the heat kernel imply
$$ \int_{\Omega}{ q_t(x,z) dx} \lesssim \sqrt{t}$$
where the implicit constant depends again only on the dimension.
Arguing as above gives that
$$ \int_{\Omega}{ f(x) dx}  \lesssim  c_n^{} |\Omega|^{-\frac{n-2}{n}} t^{\frac{n-1}{2}} e^{c_n^2 |\Omega|^{2/n}t^{-1}} \int_{\partial \Omega}{ f(z) dz}.$$
This bound gets better if we choose $t$ as large as possible, the largest admissible choice is $t \sim \delta^2$ and this implies the result.
\end{proof}

\end{document}